\documentclass[12pt]{article}
\usepackage{amsmath,fullpage,amsthm}
\usepackage{amscd}
\usepackage{amsthm,mathtools} \usepackage{amssymb}
\usepackage{latexsym}
\usepackage{eufrak}
\usepackage{euscript}

\usepackage[T1]{fontenc}

\usepackage{epsfig}
\usepackage{tikz}
\usepackage{graphics}
\usepackage{array}
\usepackage{enumerate} \usepackage{authblk}

\newtheorem{theorem}{Theorem}[section]
\newtheorem{lemma}[theorem]{Lemma}
\newtheorem{corollary}[theorem]{Corollary}
\newtheorem{conjecture}[theorem]{Conjecture}
\newtheorem{remark}[theorem]{Remark}
\newtheorem{example}[theorem]{Example}
\newtheorem{problem}[theorem]{Problem}
\theoremstyle{definition}

\theoremstyle{proof}
\theoremstyle{observation}

\begin{document}
  \setcounter{Maxaffil}{2}
   \title{On the spectrum and energy of Seidel matrix\\ for chain graphs}
   \author[,a]{Santanu Mandal\thanks{santanu.vumath@gmail.com}}
   \author[,a]{Ranjit Mehatari\thanks{ranjitmehatari@gmail.com, mehatarir@nitrkl.ac.in}}
   \author[,b]{ Kinkar Chandra Das\thanks{kinkardas2003@gmail.com}}
   \affil[a]{Department of Mathematics,}
   \affil[ ]{National Institute of Technology Rourkela,}
   \affil[ ]{Rourkela - 769008, India}
   \affil[ ]{ }
   \affil[b]{Department of Mathematics,}
   \affil[ ]{Sungkyunkwan University,}
   \affil[ ]{Suwon 16419, Republic of Korea}
   \maketitle

\begin{abstract}

We study various spectral properties of the Seidel matrix $S$ of a connected chain graph. We prove that $-1$ is always an eigenvalue of $S$ and all other eigenvalues of $S$ can have multiplicity at most two. We obtain the multiplicity of the Seidel eigenvalue $-1$, minimum number of distinct eigenvalues, eigenvalue bounds, characteristic polynomial, lower and upper bounds of Seidel energy of a chain graph. It is also shown that the energy bounds obtained here work better than the bounds conjectured by Haemers. We also obtain the minimal Seidel energy for some special chain graphs of order $n$. We also give a number of open problems.
\end{abstract}
\section{Introduction}
A chain graph can be defined in terms of induced forbidden subgraphs.
Let $C_n$ and $K_n$ denote the cycle and the complete graph on $n$ vertices, respectively. A graph with no induced subgraph isomorphic to $C_3,$ $C_5$ or $2K_2$ is called a chain graph.  Very interestingly a chain graph can be represented by a unique binary string (sequence). Starting from a single vertex, one can construct a chain graph  by repeated applications of following two operations:
\begin{enumerate}
\item
adding an isolated vertex.
\item
adding a vertex dominating to all previously added isolated vertices only.
\end{enumerate}

 We represent a chain graph $G$ on $n$ vertices using a binary string $b=\alpha_1\alpha_2\alpha_3 \cdots \alpha_n$ of length $n$,
where $\alpha_i =0$ if the vertex $v_i$ is added as an isolated vertex, and $\alpha_i =1$ if the vertex $v_i$ is dominating to all the previously added isolated vertices. As we start from a single vertex we always take $\alpha _1=0$, and since we are considering connected graphs only so $\alpha_n=1$. For any $n$ vertex chain graph, without loss of generality, we represent it by the binary string $b=0^{s_1} 1^{t_1} 0^{s_2} \ldots 0^{s_k} 1^{t_k}$, where $s_i,~t_i \geq 1$ and $\sum\limits^k_{i=1}\,(s_i+t_i)=n$. A chain graph is a bipartite graph, and it is complete bipartite if and only if $k=1$.
\medskip

Let $G$ be a simple, connected, finite graph with vertex set $V=\{v_1,v_2,\ldots,v_n\}$. Let $A$ be the $(0,1)$-adjacency  matrix \cite{Brouwer,Godsil} of $G$. Then the Seidel matrix \cite{Akbari 1,Akbari 2,AID,Berman,Aksari,Mandal, Oboudi 1,Oboudi 2,Ramane} is a square matrix of order $n$ defined by $$S=J-I-2A,$$
where $J$ is all $1$ matrix and $I$ is the identity matrix. In other words, if $s_{ij}$ is the $(i,j)$-th entry of $S$, then
$$s_{ij}=\begin{cases}
-1&\text{if }v_i\sim v_j,\\
1&\text{if }v_i\nsim v_j,\ i\neq j,\\
0&\text{if }i=j.
\end{cases}
$$

Obviously all eigenvalues of $S$ are real. We use $\lambda_1$ to denote the largest Seidel eigenvalue of graph $G$. The Seidel energy $SE(G)$ of a graph $G$ is defined as the sum of absolute values of the eigenvalues of Seidel matrix $S$. For several interesting properties of $SE(G)$, we refer \cite{Akbari 1,Ghorbani 1, Haemers, Oboudi 1,Oboudi 2}.\\

The spectral properties of adjacency and Laplacian matrix of chain graphs is well studied in literature \cite{AAD,Andelic 1,Andelic 2,Bell 1,Bell 2,Bhattacharya,Andelic 3, Das 1}. One of the most interesting property of chain graphs is that, among all connected bipartite graphs of prescribed order and size, the graph with minimal least eigenvalue (equivalently, the maximal spectral radius) is a chain graph. Around the same time but separately in $2008,$ Bhattacharya et al.~\cite{Bhattacharya} and Bell et al. \cite{Bell 1,Bell 2} first noticed this significant property of the adjacency matrix of a chain graph. After that chain graphs gained lot of focus for the past few years. The spectral properties of a chain graph have been extensively studied earlier with respect to adjacency matrix and Laplacian matrix. An\dj eli\'c et al. \cite{Andelic 1,Andelic 2,Andelic 3} published several papers related to spectra of chain graphs with key focus on eigenvalue location, characteristic polynomial and energy. Very recently, Mei et al. \cite{Mei} give some bounds of the energy and Laplacian energy of chain graphs.\\

 In this paper we consider the Seidel matrix  $S$ of a connected $n$-vertex chain graph $G$. The study on Seidel matrix of a chain graph have never been done before. Let $b=0^{s_1} 1^{t_1} 0^{s_2} \ldots 0^{s_k} 1^{t_k}$ be the binary representation of $G$. We produce a set of $n-2k+1$ orthogonal Seidel eigenvectors corresponding to the eigenvalue $-1$. We derive a formula for the Seidel characteristic polynomial of a chain graph in terms of its binary string. We see that determinant of symmetric tridiagonal matrix plays important role for evaluating the characteristic polynomial. As a consequence we obtain a formula for the determinant of the Seidel matrix of a chain graph.  More importantly we provide some interesting properties of Seidel energy of a chain graph. In \cite{Haemers}, Haemers proposed a conjecture  for any $n$ vertex graph $G$:
 \begin{lemma} {\rm\cite{Haemers}}
Let $G$ be a graph with $n$ vertices. Then $SE(G) \leq n\sqrt{n-1}$ with equality if and only if $G$ is a conference graph.
\end{lemma}

\begin{lemma}[\textbf{Haemers' Conjecture}]
Let $G$ be a graph with $n$ vertices. Then $SE(G) \geq 2(n-1)$ with equality if and only if $G$ is the complete graph $K_n$.
\end{lemma}

Ghorbani \cite{Ghorbani 1} proved this conjecture partially. Very recently Akbari et al.~\cite{Akbari 1} prove this conjecture.
Thus for a graph with $n$ vertices the known Seidel energy bound is given by
\begin{equation}
\label{Haemers}
 2(n-1)\leq SE(G)\leq n\sqrt{n-1}.
\end{equation}
Here we provide some new bounds for the Seidel energy of a chain graph. We also observe that the bounds obtained in this paper work better than the bounds in the Haemers' conjecture. We provide suitable examples to justify our results.\medskip

The contents of the paper is as follows: In Section $2$, we study few spectral properties of Seidel matrix and its quotient matrix. In particular, we prove that $-1$ is always an eigenvalue of $S$ and all other eigenvalues of the Seidel matrix can have multiplicity at most two. Moreover, we give a lower bound on the largest Seidel eigenvalue of chain graph in terms of order $n$ and $k$. In Section $3$, we obtain a formula for finding the characteristic polynomial and determinant of $S$. Section $4$ consists of Seidel energy of the chain graph. Here we obtain a better bound than Haemers' bound. Finally, we obtain the minimal Seidel energy for some special chain graphs of order $n$ in this section.


\section{Seidel eigenvalues of chain graphs}
In this section, we discuss some properties on the eigenvalues of the quotient matrix corresponding to an equitable partition of the Seidel matrix. Using these properties we obtain the multiplicity of the Seidel eigenvalue $-1$. Moreover, we prove that the spectrum (distinct eigenvalues) of the Seidel matrix of a chain graph is equal to the spectrum of that quotient matrix. In particular, if $\lambda$ is an eigenvalue which is not an eigenvalue of that quotient matrix, then $\lambda=-1$.
\subsection{Quotient Matrix}
Let us consider a chain graph $G$ with the binary string $b=0^{s_1} 1^{t_1} 0^{s_2} \ldots 0^{s_k} 1^{t_k}$. Then the  Seidel matrix $S$ of $G$ is a square matrix of size $n$, given by

$$S=\begin{bmatrix}
(J-I)_{s_1}&-J_{s_1\times t_1}&J_{s_1\times s_2}&-J_{s_1\times t_2}&J_{s_1\times s_3}&-J_{s_1\times t_3}&\ldots&-J_{s_1\times t_k}\\[1mm]
-J_{t_1\times s_1}&(J-I)_{t_1}&J_{t_1\times s_2}&J_{t_1\times t_2}&J_{t_1\times s_3}&J_{t_1\times t_3}&\ldots&J_{t_1\times t_k}\\[1mm]
J_{s_2\times s_1}&J_{s_2\times t_1}&(J-I)_{s_2}&-J_{s_2\times t_2}&J_{s_2\times s_3}&-J_{s_2\times t_3}&\ldots&-J_{s_2\times t_k}\\[1mm]
-J_{t_2\times s_1}&J_{t_2\times t_1}&-J_{t_2\times s_2}&(J-I)_{t_2}&J_{t_2\times s_3}&J_{t_2\times t_3}&\ldots&J_{t_2\times t_k}\\[1mm]
& & & & & & \ddots \\[1mm]
J_{s_k\times s_1}&J_{s_k\times t_1}&J_{s_k\times s_2}&J_{s_k\times t_2}&J_{s_k\times s_3}&J_{s_k\times t_3}&\ldots &-J_{s_k\times t_k}\\[1mm]
-J_{t_k\times s_1}&J_{t_k\times t_1}&-J_{t_k\times s_2}&J_{t_k\times t_2}&-J_{t_k\times s_3}&J_{t_k\times t_3}&\ldots&(J-I)_{t_k}
\end{bmatrix},$$
where  $J_{m\times n}$ is all 1 block matrix of size $m\times n$ and the diagonal blocks of $S$ are the square matrices of size $s_1\times s_1,~t_1\times t_1,~s_2\times s_2,~t_2\times t_2,~\ldots,~t_k\times t_k$. \\

We now construct an equitable partition in the following way. For the representation $b=0^{s_1} 1^{t_1} 0^{s_2} \ldots 0^{s_k} 1^{t_k}$ of $G$, let $V_{s_1}$ denote the set of vertices representing first $s_1$ added vertices, $V_{t_1}$ denote the set of vertices representing next $t_1$ added vertices, $V_{s_2}$ denote the set of vertices representing next $s_2$ added vertices and finally $V_{t_k}$ denote the set of vertices representing last $t_k$ added vertices. Now we consider the vertex partition of $\pi =\{ C_1,C_2,C_3,\ldots,C_{2k}\}$, where   $C_i=V_{s_j}=s_j$, if $i=2j-1$ $(j=1,\,2,\ldots,\,k)$, \text{and} $C_i=V_{t_j}=t_j$, if $i=2j$ $(j=1,\,2,\ldots,\,k)$. Then $\pi$ is an equitable partition of $G$ of size $2k$. Throughout this paper we consider this equitable partition only, and by writing ``$G$ is a chain graph of order $n$ with $|\pi|=2k$'' we mean that $G$ is a chain graph graph of order $n$ with the binary string $b=0^{s_1} 1^{t_1} 0^{s_2} \ldots 0^{s_k} 1^{t_k}$. Let $Q_S$ be the corresponding quotient matrix. Then  $Q_S$ is a square matrix of size $2k$, given by
 $$Q_S=\begin{bmatrix}
(s_1 -1) & -t_1 &s_2  &-t_2&s_3  &-t_3& \ldots & -t_k \\[1mm]
-s_1&(t_1 -1) & s_2 &t_2&s_3&t_3& \ldots & t_k\\[1mm]
s_1 &t_1&(s_2-1)&-t_2 &s_3 &-t_3& \ldots & -t_k\\[1mm]
-s_1&t_1&-s_2&(t_2 -1)&s_3&t_3&\ldots&t_k\\[1mm]
s_1&t_1&s_2&t_2&(s_3-1)&-t_3&\ldots&-t_k\\[1mm]
-s_1&t_1&-s_2&t_2&-s_3&(t_3-1)&\ldots& t_k\\[1mm]
& & & & & & \ddots \\[1mm]
s_1&t_1&s_2&t_2&s_3&t_3&\ldots&-t_k \\[1mm]
-s_1 &  t_1 &-s_2 &t_2 &-s_3 &t_3& \ldots & (t_k -1)
\end{bmatrix}.$$

Let $D=diag[s_1,t_1,s_2,t_2,\ldots,s_k,t_k]$. Then $Q_S$ is similar to the symmetric matrix $D^{\frac{1}{2}}Q_SD^{-\frac{1}{2}}$, thus $Q_S$ is diagonalizable.
Let $\lambda$ be an eigenvalue of $Q_S$ with corresponding eigenvector $X\in\mathbb{R}^{2k}$. Let $P$ be the characteristic matrix (of order $n\times 2k$) for the equitable partition $\pi=\{C_1,C_2,\ldots,C_{2k}\}$, i.e., $(i,j)$-th entry of $P$ is 1 if $i\in C_j$ and 0 otherwise. Then it is easy to verify that $SP=PQ_S$. Then $S(PX)=\lambda (PX)$. Which implies that every eigenvalue of $Q_S$ is also an eigenvalue of $S$. It is easy to verify that trace of the quotient matrix $Q_S$ is $n-2k$. Then trace of $Q_S^2$ can be obtained  by the formula $$trace\ Q_S^2=(trace\  Q_S)^2-2\sum_{i<j}\lambda_i\lambda_j,$$
where $\lambda_i$'s are eigenvalues of $Q_S$. We obtain the following:
\begin{lemma}
\label{Chain_Quotient_Lm1}
Let $G$ be an chain graph of order $n$ with $|\pi|=2k$. Then
$$trace\ Q_S^2=n^2-2n+2k. $$
\end{lemma}

\begin{proof} Let $\lambda_1,\lambda_2,\ldots,\lambda_{2k}$ be eigenvalues of $Q_S$. Since $\sum\limits^k_{i=1}\,(s_i+t_i)=n$, we obtain
\begin{align*}
trace\ Q_S^2&=\sum_{i=1}^{2k}\lambda_i^2\\
&=\sum\limits^k_{i=1}\,\Big[(s_i-1)^2+s_i\,\Big(\sum\limits^k_{j=1,j\neq i}\,s_j+\sum\limits^k_{j=1}\,t_j\Big)\Big]+\sum\limits^k_{i=1}\,\Big[(t_i-1)^2+t_i\,\Big(\sum\limits^k_{j=1}\,s_j+\sum\limits^k_{j=1,j\neq i}\,t_j\Big)\Big]\\
&=\sum\limits^k_{i=1}\,\Big[(s_i-1)^2+s_i\,(n-s_i)\Big]+\sum\limits^k_{i=1}\,\Big[(t_i-1)^2+t_i\,(n-t_i)\Big]\\
&=\sum\limits^k_{i=1}\,\Big[s_i\,(n-2)+t_i\,(n-2)+2\Big]=n^2-2n+2k.
\end{align*}
\end{proof}

\begin{theorem}
\label{Chain_quotient_th1}
Let $b=0^{s_1} 1^{t_1} 0^{s_2} \ldots 0^{s_k} 1^{t_k}$ be the binary string of a chain graph $G$. Then $-1$ is a simple eigenvalue of $Q_S$.
\end{theorem}

\begin{proof}
The eigenvector corresponding to the eigenvalue $-1$ is $X=\left[\begin{array}{ccccccc}
t_k&0&0&0&\cdots&0&s_1
\end{array}\right]^T$. Clearly the geometric multiplicity of the eigenvalue $-1$ is one. $Q_S$ being diagonalizable, the algebraic multiplicity of the eigenvalue $-1$ is also one. Hence the result follows.
\end{proof}



\begin{theorem}
\label{Chain_quotient_th2}
The eigenvalues of  $Q_S$ can have multiplicity at most two.
\end{theorem}
\begin{proof}
Suppose $\lambda$ is an eigenvalue of $Q_S$. Let $X=\left[\begin{array}{ccccc}
x_1&x_2&x_3&\cdots&x_{2k}
\end{array}\right]^T$ be an eigenvector corresponding to $\lambda$. We already proved that $\lambda=-1$ is a  simple eigenvalue. Now we have to prove the theorem for $\lambda \neq -1$.
Then from the relation $Q_SX=\lambda X$, we have $2k$ linear equations given by
\begin{equation}
\tag{A}
\begin{aligned}
(s_1 -1)x_1-t_1x_2+s_2x_3 -t_2x_4+\ldots-t_kx_{2k}&=\lambda x_1,\\
-s_1x_1+(t_1-1)x_2+s_2x_3+t_2x_4+\ldots t_kx_{2k}&=\lambda x_2,\\
s_1x_1+t_1x_2 +(s_2-1)x_3-t_2x_4+\ldots -t_kx_{2k}&=\lambda x_3,\\
-s_1x_1+t_1x_2-s_2x_3+(t_2-1)x_4+\ldots +t_kx_{2k}&=\lambda x_4,\\
s_1x_1+t_1x_2+s_2x_3+t_2x_4+\ldots -t_kx_{2k}&=\lambda x_5,\\
-s_1x_1+t_1x_2-s_2x_3+t_2x_4+\ldots +t_kx_{2k}&=\lambda x_6,\\
\ldots~~\ldots~~\ldots~~\ldots~~~\ldots~~\ldots~~\ldots&\vdots\\
s_1x_1+t_1x_2+s_2x_3+t_2x_4+\ldots -t_kx_{2k}&=\lambda x_{2k-1},\\
-s_1x_1+t_1x_2-s_2x_3+t_2x_4+\ldots +(t_k-1)x_{2k}&=\lambda x_{2k}.
\end{aligned}
\end{equation}

Now for odd $j\geq3$, subtracting $j$-th equation from first equation, we get,
\begin{equation}
\label{Chain_multi_eq1}
 x_j=x_1+\frac{2t_1}{1+\lambda}x_2+\frac{2t_2}{1+\lambda}x_4+\cdots +\frac{2t_{\frac{j-1}{2}}}{1+\lambda}x_{j-1},
\end{equation}
and for even $i\geq4$, subtracting $j$-th equation from first equation, we get
\begin{equation}
\label{Chain_multi_eq2}
x_i=x_2-\frac{2s_2}{1+\lambda}x_3-\frac{2s_3}{1+\lambda}x_5-\cdots -\frac{2s_{\frac{i}{2}}}{1+\lambda}x_{i-1}.
\end{equation}
Putting $j=3$ in (\ref{Chain_multi_eq1}), we get
 $$x_3=x_1+\frac{2t_1}{1+\lambda}x_2=a_3x_1+b_3x_2, \text{(say)}.$$
Putting $i=4$ in (\ref{Chain_multi_eq2}) and using the expression of $x_3$, we get
 $$x_4=x_2-\frac{2s_2}{1+\lambda}x_3=x_2-\frac{2s_2}{1+\lambda}(a_3x_1+b_3x_2)=a_4x_1+b_4x_2, \text{(say)}.$$
Putting $j=5$ in (\ref{Chain_multi_eq1}) and using the expression of $x_4$, we get
 $$x_5=x_1+\frac{2t_1}{1+\lambda}x_2+\frac{2t_2}{1+\lambda}x_4=a_5x_1+b_5x_2, (say).$$
Putting $i=6$ in (\ref{Chain_multi_eq2}) and using the expressions of $x_3$ and $x_5$, we get
 $$x_6=x_2-\frac{2s_2}{1+\lambda}x_3-\frac{2s_3}{1+\lambda}x_5=a_6x_1+b_6x_2, (say).$$
Proceeding in this way, iteratively, we can obtain real numbers $a_3,a_4, \ldots a_{2k}$ and $b_3, b_4, \ldots b_{2k}$ such that
\begin{equation}
\label{Chain_multi_eq3}
 x_i=a_ix_1+b_ix_2~~~\mbox{ for }3\leq i\leq 2k.
\end{equation}

\vspace*{3mm}

Assume to the contrary that the multiplicity of the eigenvalue $\lambda\,(\neq -1)$ of $Q_S$ is at least three. Then we can assume that $Y=\left[\begin{array}{cccccc}
0&0&y_3&y_4&\cdots&y_{2k}
\end{array}\right]^T$ is an eigenvector corresponding to $\lambda$ of $Q_S$. By (\ref{Chain_multi_eq3}), we obtain $y_3=y_4=\cdots=y_{2k}=0$, that is, $Y={\bf 0}$, a contradiction as $Y$ is an eigenvector. This completes the proof of the theorem.
\end{proof}

 Since $-1$ is a simple eigenvalue of $Q_S$, therefore among $(2k-1)$ remaining eigenvalues, the minimum number of distinct eigenvalues is $(\frac{2k-2}{2}+1)$.
 As a consequence of this fact we obtain here a straightforward result which counts the minimum number of distinct eigenvalues of $Q_S$.
\begin{theorem}
\label{Chain_quotient_th3}
$Q_S$ has at least $(k+1)$ distinct eigenvalues.
\end{theorem}
  The quotient matrix  $Q_S$ has $2k$ eigenvalues and with at least $k+1$ distinct eigenvalues. We did not find any chain graph for which the number of distinct positive and the distinct negative eigenvalues of $S$ are unequal.

 \begin{problem}
 Is there exists a chain graph for which the number of positive eigenvalues of $Q_S$ is not equal to the number of negative eigenvalues?
 \end{problem}
\subsection{Eigenvalue bounds and multiplicity of the eigenvalue $-1$ }
Let us consider a chain graph $G$ with the binary string $0^{s_1} 1^{t_1} 0^{s_2} \ldots 0^{s_k} 1^{t_k}$. Let $n_{-1}(S)$ denote the multiplicity of the eigenvalue $-1$ of the matrix $S$.  We now derive a formula for  $n_{-1}(S)$. For this first we construct eigenvectors corresponding to $-1$ which does not belong to spectrum of $Q_S$.
\medskip

For $i>1$, we define a set $\{E_j^i\}$ of $i-1$ orthogonal row-vectors in $\mathbb{R}^i$ by
$$E_j^i=\textbf{\emph{e}}_1(i)+\textbf{\emph{e}}_2(i)+\cdots+\textbf{\emph{e}}_{j}(i)-j\textbf{\emph{e}}_{j+1}(i)\ \forall 1\leq j\leq i-1,$$
where the row vector $\textbf{\emph{e}}_{j}(i)$ represents the $j$-th standard basis element of $\mathbb{R}^i$. \\

Now for each $s_i\geq 2$, we define
$$X_{s_i}(j)=[O_{s_1}\ O_{t_1}\ \cdots\ O_{t_{i-1}}\ E_j^{s_i} \ O_{t_{i}}\ \cdots\ O_{t_k}]^T,\ \  1\leq j \leq s_i-1,\,1\leq i\leq k,$$
where row-vector $O_r$ denote the zero vector in $\mathbb{R}^r$.
Then the set $\{X_{s_i}(1),X_{s_i}(2),\ldots, X_{s_i}(s_i-1)\}$ is orthogonal and $SX_{s_i}(j)=-X_{s_i}(j)$ for all $1\leq j\leq s_i-1$. There for each $s_i>1$ $(1\leq i\leq k)$, the constructed set contains $s_i-1$ orthogonal eigenvectors corresponding to $-1$.\\

Again for each $t_i\geq 2$, define
$$Y_{t_i}(j)=[O_{s_1}\ O_{t_1}\ \cdots\ O_{s_{i}}\ E_j^{t_i} \ O_{s_{i+1}}\ \cdots\ O_{t_k}]^T,\ \ 1\leq j \leq t_i-1,\,1\leq i\leq k.$$
As above, for each $t_i>1$ $(1\leq i\leq k)$, the set $\{Y_{t_i}(1),Y_{t_i}(2),\ldots, Y_{t_i}(t_i-1)\}$ contains $t_i-1$ orthogonal eigenvectors corresponding to $-1$.\\

Each of $X_{s_i}(j)$'s and $Y_{t_i}(j)$'s has column sum zero in each of the vertex partition. Now let $\lambda$ be an eigenvalue of $Q_S$ with eigenvector $X =\left[\begin{array}{ccccc}
x_1&x_2&x_3&\cdots&x_{2k}
\end{array}\right]^T\in\mathbb{R}^{2k}$. Then $PX$ is an eigenvector corresponding to the eigenvalue $\lambda$ for the Seidel matrix S and the $j$-th entry of the eigenvector is $x_i$ for every $j\in C_i$, i.e., the eigenvector $PX$ corresponding to the eigenvalue $\lambda$ is constant in each vertex set in $\pi$. Therefore $PX$ is orthogonal to each of these $X_{s_i}(j)$'s and $Y_{t_i}(j)$'s. Using this fact we now compute the multiplicity of the eigenvalue $-1$.

\begin{theorem}
\label{Chain_pm1_Ch1} Let $G$ be a chain graph of order $n$ with $|\pi|=2k$. Then
$$n_{-1}(S)= n-2k+1.$$
\end{theorem}

\begin{proof}
We already observed that, if $s_i\geq 2$ $(1\leq i\leq k)$, then the set $\{X_{s_i}(1),X_{s_i}(2),\ldots, X_{s_i}(s_i-1)\}$ contains $s_i-1$ orthogonal eigenvectors corresponding to $-1$. Now for $s_{\ell}, \ s_m\geq2$, the vectors $X_{s_{\ell}}(j)$ and $X_{s_m}(k)$ are orthogonal for all $1\leq j< s_{\ell}$  and $1\leq k<s_m$. Therefore the set
$$\Big\{X_{s_i}(j)\,|s_i>1,\,1\leq j \leq s_i-1,\,1\leq i\leq k\Big\}$$ is a set of $\sum s_i-k$ orthogonal eigenvectors corresponding to $-1$.\\
By a similar argument, the set
$$\Big\{Y_{t_i}(j)\,|t_i>1,\,1\leq j \leq t_i-1,\,1\leq i\leq k\Big\}$$
provides $\sum t_i-k$ orthogonal eigenvectors corresponding to $-1$. Again any eigenvector of the form $X_{s_i}(\ell)$ is orthogonal to any eigenvector of the form $Y_{t_j}(m)$. Since $Q_S$ has a simple eigenvalue $-1$, therefore the multiplicity of the Seidel eigenvalue $-1$ is exactly $\sum s_i-k+\sum t_i-k+1$. That means,
$$n_{-1}(S)=n-2k+1 .$$
This completes the proof of the theorem.
\end{proof}
From Theorem \ref{Chain_quotient_th1} and Theorem \ref{Chain_pm1_Ch1}, it is clear that if $\lambda\neq-1$ is an eigenvalue of $S$ then it also an eigenvalue of $Q_S$. Since the eigenvalues of $Q_S$ can have multiplicity at most two. Thus by Theorem \ref{Chain_quotient_th3} and Theorem \ref{Chain_pm1_Ch1} we have the following corollary.

\begin{corollary}\label{Chain_quotient_cor1} Let $G$ be a chain graph with $|\pi|=2k$. If $m(G)$ denotes the number of distinct Seidel eigenvalues, then $m(G)$ satisfies the following:
$$k+1\leq m(G) \leq 2k.$$
\end{corollary}

\begin{remark}
Here both the bounds are tight. Let $G_1$ and $G_2$ be chain graphs with binary  strings $b_1=0^11^{2} 0^{2} 1^20^{2} 1^1$  and $b_2=0^11^10^11^10^11^1$, respectively. Then $G_1$ has 4 distinct Seidel eigenvalues (i.e., $m(G_1)=k+1$), on the other hand $G_2$ has $6$ distinct Seidel eigenvalues (i.e., $m(G_2)=2k$).
\end{remark}

We have checked several examples and we found that for any chain graph $G$, either $m(G)=k+1$ or $m(G)=2k$.
\begin{problem}
Is there exist a chain graph for which $k+1<m(G)<2k$?
\end{problem}\medskip

We now give some upper and lower bounds on the largest Seidel eigenvalue of a chain graph. Let $X=\left[\begin{array}{ccccc}
x_1&x_2&x_3&\cdots&x_{2k}
\end{array}\right]^T$ be an eigenvector corresponding to the largest Seidel eigenvalue $\lambda_1(G)$ of $Q_S$. Then $Q_SX=\lambda_1(G)X$, that is,
   $$\lambda_1(G)x_i=\sum\limits^{2k}_{j=1}\,(Q_S)_{ij}x_j,~~i=1,\,2,\ldots,\,2k.$$
From the above, we obtain $\lambda_1(G)\leq \sum\limits^{2k}_{j=1}\,|(Q_S)_{ij}|=\sum\limits^{k}_{i=1}\,s_i+\sum\limits^{k}_{i=1}\,t_i-1=n-1$. Moreover, one can easily see that the equality holds if and only if $k=1$, that is, if and only if $G$ is isomorphic to a complete bipartite graph.
We now give a lower bound for the largest eigenvalue $\lambda_1(G)$ using the Rayleigh quotient.
\begin{theorem}\label{1w1} Let $G$ be a chain graph of order $n$ with $|\pi|=2k$. Then
$$\lambda_1(G)\geq\frac{n}{2}+1-\frac{2k}{n}.$$
\end{theorem}

\begin{proof} For any symmetric matrix $B$, we have
$$\lambda_1(B)=\max_{X\neq 0} \frac{X^TBX}{X^TX}.$$
The matrix $Q_S$ is similar to the symmetric matrix $Q_S^{'}$ given by
 $$Q_S^{'}=\left[\begin{array}{ccccc|ccccc}
s_1 -1 & \sqrt{s_1s_2} &\sqrt{s_1s_3}&\cdots&\sqrt{s_1s_k} &-\sqrt{s_1t_1}&-\sqrt{s_1t_2}&-\sqrt{s_1t_3}&\cdots &-\sqrt{s_1t_k} \\[1mm]
\sqrt{s_2s_1}&s_2-1&\sqrt{s_2s_3} & \cdots&\sqrt{s_2s_k} &\sqrt{s_2t_1}&-\sqrt{s_2t_2}&-\sqrt{s_2t_3}&\cdots &-\sqrt{s_2t_k}\\[1mm]
\sqrt{s_3s_1}&\sqrt{s_3s_2}&s_3-1& \cdots&\sqrt{s_3s_k} &\sqrt{s_3t_1}&\sqrt{s_3t_2}&-\sqrt{s_1t_3}&\cdots &-\sqrt{s_1t_k}\\[1mm]
\cdots&\cdots&\cdots&\ddots&\cdots&\cdots&\cdots&\cdots&\cdots&\cdots\\[1mm]
\sqrt{s_ks_1}&\sqrt{s_ks_2}&\sqrt{s_ks_2}& \cdots&s_k-1 &\sqrt{s_kt_1}&\sqrt{s_kt_2}&\sqrt{s_kt_3}&\cdots &-\sqrt{s_kt_k}\\[1mm]
\hline
-\sqrt{s_1t_1}&\sqrt{s_2t_1}&\sqrt{s_3t_1}&\cdots&\sqrt{s_kt_1} &t_1-1&\sqrt{t_1t_2}&\sqrt{t_1t_3}&\cdots&\sqrt{t_1t_k}\\[1mm]
-\sqrt{s_1t_2}&-\sqrt{s_2t_2}&\sqrt{s_3t_2}&\cdots&\sqrt{s_kt_2} &\sqrt{t_1t_2}&t_2-1&\sqrt{t_2t_3}&\cdots&\sqrt{t_2t_k}\\[1mm]
-\sqrt{s_1t_3}&-\sqrt{s_2t_3}&-\sqrt{s_3t_3}&\cdots&\sqrt{s_kt_3} &\sqrt{t_1t_3}&\sqrt{t_2t_3}&t_3-1&\cdots&\sqrt{t_3t_k}\\[1mm]
\cdots&\cdots&\cdots&\cdots&\cdots&\cdots&\cdots&\cdots&\ddots&\cdots\\
-\sqrt{s_1t_k}&-\sqrt{s_2t_k}&-\sqrt{s_3t_k}&\cdots&-\sqrt{s_kt_k} &\sqrt{t_1t_k}&\sqrt{t_2t_k}&\sqrt{t_3t_k}&\cdots&t_k-1\\
\end{array}\right].$$
Let $X=[{\sqrt{s_1}}\  {\sqrt{s_2}}\  {\sqrt{s_3}}\ \cdots \ {\sqrt{s_k}} \ {-\sqrt{t_1}} \ {-\sqrt{t_2}}\ -\sqrt{t_3}\ \cdots\  -{\sqrt{t_k}}]^T$.
Here $X^TX=n$ and
\begin{align}
X^TQ_S^{'}X&=\sum_{i=1}^ks_i^2+2\sum_{\substack{1\leq i<j\leq k}}s_is_j-\sum_{i=1}^ks_i+\sum_{i=1}^kt_i^2+2\sum_{\substack{1\leq i<j\leq k}}\,t_it_j-\sum_{i=1}^kt_i+2\sum_{i=1}^ks_it_i\nonumber\\
&=\Big(\sum_{i=1}^k\,s_i\Big)^2+\Big(\sum_{i=1}^k\,t_i\Big)^2-\sum_{i=1}^k(s_i+t_i)+2\sum_{i=1}^ks_it_i.\label{kin1}
\end{align}
It is well known that $a^2+b^2\geq \dfrac{(a+b)^{2}}{2}$ for all $a,\,b\geq 0$. Thus we have
   $$\Big(\sum_{i=1}^ks_i\Big)^2+\Big(\sum_{i=1}^kt_i\Big)^2\geq \frac{1}{2}\,\Big(\sum_{i=1}^k\,s_i+\sum_{i=1}^k\,t_i\Big)^2=\frac{1}{2}\,n^2.$$
Now since the chain graph has $n$ vertices and ${s_i}\geq1$ and ${t_i}\geq1$, we obtain
\begin{align*}
2\sum_{i=1}^k\,s_it_i-\sum_{i=1}^k\,(s_i+t_i)&=\sum_{i=1}^k\,\Big[s_i\,(t_i-1)+t_i\,(s_i-1)\Big]\\
&\geq \sum_{i=1}^k\,(s_i+t_i-2)=n-2k.
\end{align*}
Using the above results, from (\ref{kin1}), we  obtain
$$X^TQ_S^{'}X\geq\frac{n^2}{2}+n-2k.$$
Which gives,
$$\lambda_1(G)=\lambda_1(Q_S)=\lambda_1(Q^{'}_s)\geq \frac{X^TQ_S^{'}X}{X^TX}\geq \frac{n}{2}+1-\frac{2k}{n}.$$
This completes the proof of the theorem.
\end{proof}


\section{Characteristic polynomial of chain graphs}
In this section, we obtain a formula for finding the characteristic polynomial of the Seidel matrix $S$ of a connected chain graph using its binary representation. First we convert the quotient matrix $Q_{J-2A}$ into a sparse matrix by using row and column operations. The characteristic polynomials of $Q_S$ and $Q_{J-2A}$ satisfy the following relation:
\begin{equation}
\label{Chain_cp_eq0}
\psi_{Q_S}(x)=\det(Q_S-xI)=\det\Big(Q_{J-2A}-(x+1)I\Big)=\psi_{Q_{J-2A}}(x+1)
\end{equation}
where $\det B$ stands for determinant of a square matrix $B$.
\begin{theorem}
\label{Chain_cp_th1}
The characteristic polynomial $\psi_S(x)$ of the chain graph $G$ with binary string $b=0^{s_1} 1^{t_1} 0^{s_2} \ldots 0^{s_k} 1^{t_k}$, $k\geq2$, is given by the following formula,
\begin{equation}
\label{Chain_cp_eq1}
 \psi_S(x)= \frac{(x+1)}{2}^{n-2k+1} \Big[2(x+1)^{2k-1}+(-1)^{k-1}(x+1)^2\det U_{x+1}+(-1)^k\det V_{x+1}\Big]
\end{equation}
where $U_{x+1}$ is a tridiagonal matrix of order $2k-3$ given by

 $$U_{x+1}=\begin{bmatrix}
2s_2& x+1 &0&0  &0& \ldots & 0&0 \\[1mm]
x+1 &2 t_2 &x+1 &0&0& \ldots &0&0\\[1mm]
0& x+1 &2s_3 & x+1 &0& \ldots &0&0\\[1mm]
0&0& x+1 &2t_3& x+1 &\ldots&0&0\\[1mm]
0&0&0& x+1 &2s_4& \ldots &0&0 \\[1mm]
& & & & & \ddots \\[1mm]
0&0&0&0&0&\ldots&2t_{k-1}& x+1\\[1mm]
0&0 &0 &0 &0 & \ldots & x+1 &2s_k
\end{bmatrix}$$
and $V_{x+1}$ is another tridiagonal matrix of order $2k-1$ given by
 $$V_{x+1}=\begin{bmatrix}
2t_1& x+1 &0&0  &0& \ldots & 0&0 \\[1mm]
x+1 &2 s_2 & x+1 &0&0& \ldots &0&0\\[1mm]
0& x+1&2t_2 & x+1 &0& \ldots &0&0\\[1mm]
0&0& x+1 &2s_3& x+1&\ldots&0&0\\[1mm]
0&0&0& x+1 &2t_3& \ldots &0&0 \\[1mm]
& & & & & \ddots \\[1mm]
0&0&0&0&0&\ldots&2s_k& x+1\\[1mm]
0&0 &0 &0 &0 & \ldots & x+1 &2(s_1+t_k)
\end{bmatrix}.$$
\end{theorem}

\begin{proof}
We have $\psi_{Q_S}(x)=\psi_{Q_{J-2A}}(x+1)$. We apply some elementary row and column operations on the quotient matrix $Q_{J-2A}$ given by
 $$Q_{J-2A}=\begin{bmatrix}
s_1 & -t_1 &s_2  &-t_2&s_3  &-t_3& \ldots & s_k&-t_k \\[1mm]
-s_1&t_1 & s_2 &t_2&s_3&t_3& \ldots & s_k&t_k\\[1mm]
s_1 &t_1&s_2&-t_2 &s_3 &-t_3& \ldots &s_k& -t_k\\[1mm]
-s_1&t_1&-s_2&t_2&s_3&t_3&\ldots&s_k&t_k\\[1mm]
s_1&t_1&s_2&t_2&s_3&-t_3&\ldots&s_k&-t_k \\[1mm]
-s_1&t_1&-s_2&t_2&-s_3&t_3&\ldots &s_k&t_k \\[1mm]
& & & & & & \ddots \\[1mm]
s_1&t_1&s_2&t_2&s_3&t_3&\ldots&s_k&-t_k \\[1mm]
-s_1 &  t_1 &-s_2 &t_2 &-s_3 &t_3& \ldots &-s_k & t_k
\end{bmatrix}.$$

The characteristic polynomial of $Q_{J-2A}$ is given by

\begin{eqnarray*}
\psi_{Q_{J-2A}}(x)&=&\det (Q_{J-2A}-xI)\\&=&\begin{vmatrix}
s_1-x & -t_1 &s_2  &-t_2&s_3  &-t_3& \ldots & s_k&-t_k \\[1mm]
-s_1&t_1-x & s_2 &t_2&s_3&t_3& \ldots & s_k&t_k\\[1mm]
s_1 &t_1&s_2-x&-t_2 &s_3 &-t_3& \ldots &s_k& -t_k\\[1mm]
-s_1&t_1&-s_2&t_2-x&s_3&t_3&\ldots&s_k&t_k\\[1mm]
s_1&t_1&s_2&t_2&s_3-x&-t_3&\ldots&s_k&-t_k \\[1mm]
-s_1&t_1&-s_2&t_2&-s_3&t_3-x&\ldots &s_k&t_k \\[1mm]
& & & & & & \ddots \\[1mm]
s_1&t_1&s_2&t_2&s_3&t_3&\ldots&s_k-x&-t_k \\[1mm]
-s_1 &  t_1 &-s_2 &t_2 &-s_3 &t_3& \ldots &-s_k &t_k-x
\end{vmatrix}
\end{eqnarray*}
Applying the following operations respectively,\\

$ R_i^{'}\gets R_i-R_{i-2}$, for $i=even\geq4$; $ R_j^{'}\gets R_j-R_{j-2}$, for $j=odd\geq3$; $ R_1\gets R_1+R_2$. Then
$$ \psi_{Q_{J-2A}}(x)= \begin{vmatrix}
-x &-x &2s_2  &0&2s_3  &0& \ldots & 2s_k&0 \\[1mm]
-s_1&t_1-x & s_2 &t_2&s_3&t_3& \ldots & s_k&t_k\\[1mm]
x &2t_1&-x&0 &0 &0& \ldots &0&0\\[1mm]
0&x&-2s_2&-x&0&0&\ldots&0&0\\[1mm]
0&0&x&2t_2&-x&0&\ldots&0&0 \\[1mm]
0&0&0&x&-2s_3&-x&\ldots&0&0 \\[1mm]
& & & & & & \ddots \\[1mm]
0&0&0&0&0&0&\ldots&-x&0 \\[1mm]
0&0&0&0&0&0&\ldots&-2s_k&-x
\end{vmatrix}.$$

Again applying the following operations respectively,\\

$ R_1^{'}\gets R_1+ \sum R_i$, for $i=even\geq4$; $ R_2^{'}\gets R_2-\frac{1}{2}\{ \sum R_j-\sum R_p\}$, for $j=odd\geq3$, and $p=even\geq4$, we obtain
$$ \psi_{Q_{J-2A}}(x)= \begin{vmatrix}
-x &0 &0  &0&0  &0& \ldots &0&-x \\[1mm]
-(s_1+\frac{x}{2})&-\frac{x}{2} &0 &0&0&0& \ldots &\frac{x}{2}&(t_k-\frac{x}{2})\\[1mm]
x &2t_1&-x&0 &0 &0& \ldots &0&0\\[1mm]
0&x&-2s_2&-x&0&0&\ldots&0&0\\[1mm]
0&0&x&2t_2&-x&0&\ldots&0&0 \\[1mm]
0&0&0&x&-2s_3&-x&\ldots&0&0 \\[1mm]
& & & & & & \ddots \\[1mm]
0&0&0&0&0&0&\ldots&-x&0 \\[1mm]
0&0&0&0&0&0&\ldots&-2s_k&-x
\end{vmatrix}.$$

Performing $ R_2^{'}\gets 2R_2$ and $ C_{2k}^{'}\gets C_{2k}-C_1$ respectively,

\begin{eqnarray*}
\psi_{Q_{J-2A}}(x)&=& \frac{1}{2}\begin{vmatrix}
-x &0 &0  &0&0  &0& \ldots &0&0 \\
-(2s_1+x)&-x &0 &0&0&0& \ldots &x&2(s_1+t_k)\\[1mm]
x &2t_1&-x&0 &0 &0& \ldots &0&-x\\[1mm]
0&x&-2s_2&-x&0&0&\ldots&0&0\\[1mm]
0&0&x&2t_2&-x&0&\ldots&0&0 \\[1mm]
0&0&0&x&-2s_3&-x&\ldots&0&0 \\[1mm]
& & & & & & \ddots \\[1mm]
0&0&0&0&0&0&\ldots&-x&0 \\[1mm]
0&0&0&0&0&0&\ldots&-2s_k&-x
\end{vmatrix}\\\\\\
&=& -\frac{x}{2}\,\begin{vmatrix}
-x &0 &0&0&0& \ldots &x&2(s_1+t_k)\\[1mm]
2t_1&-x&0 &0 &0& \ldots &0&-x\\[1mm]
x&-2s_2&-x&0&0&\ldots&0&0\\[1mm]
0&x&2t_2&-x&0&\ldots&0&0 \\[1mm]
0&0&x&-2s_3&-x&\ldots&0&0 \\[1mm]
& & & & & \ddots \\[1mm]
0&0&0&0&0&\ldots&-x&0 \\[1mm]
0&0&0&0&0&\ldots&-2s_k&-x
\end{vmatrix}\end{eqnarray*}
After taking permutations on rows,

\begin{eqnarray*}
\psi_{Q_{J-2A}}(x)&=&-\frac{x}{2}(-1)^{2k-2}\begin{vmatrix}
2t_1&-x&0 &0 &0& \ldots &0&-x\\[1mm]
x&-2s_2&-x&0&0&\ldots&0&0\\[1mm]
0&x&2t_2&-x&0&\ldots&0&0 \\[1mm]
0&0&x&-2s_3&-x&\ldots&0&0 \\[1mm]
0&0&0&x&2t_3&\ldots&0&0\\[1mm]
& & & & & \ddots \\[1mm]
0&0&0&0&0&\ldots&-2s_k&-x\\[1mm]
-x&0 &0&0&0& \ldots &x&2(s_1+t_k)
\end{vmatrix}\\\\\\
 &=& -\frac{x}{2}\begin{vmatrix}
2t_1&-x&0 &0 &0& \ldots &0&-x\\[1mm]
x&-2s_2&-x&0&0&\ldots&0&0\\[1mm]
0&x&2t_2&-x&0&\ldots&0&0 \\[1mm]
0&0&x&-2s_3&-x&\ldots&0&0 \\[1mm]
0&0&0&x&2t_3&\ldots&0&0\\[1mm]
& & & & & \ddots \\[1mm]
0&0&0&0&0&\ldots&-2s_k&-x\\[1mm]
-x&0 &0&0&0& \ldots &x&2(s_1+t_k)
\end{vmatrix}.\end{eqnarray*}
Expanding we obtain
\begin{equation}
\label{Chain_cp_eq3}
 \psi_{Q_{J-2A}}(x)= -\frac{x}{2}\,\Big(\det M_1-x\det M_2\Big),
\end{equation}

where $M_1$ is a square matrix of order $2k-1$ given by
$$M_1=\begin{bmatrix}
2t_1&-x &0&0  &0& \ldots & 0&0 \\[1mm]
x &-2 s_2 &-x&0&0& \ldots &0&0\\[1mm]
0&x&2t_2 &-x&0& \ldots &0&0\\[1mm]
0&0&x&-2s_3&-x&\ldots&0&0\\[1mm]
0&0&0&x&2t_3& \ldots &0&0 \\[1mm]
& & & & & \ddots \\[1mm]
0&0&0&0&0&\ldots&-2s_k&-x\\[1mm]
-x&0 &0 &0 &0 & \ldots &x&2(s_1+t_k)
\end{bmatrix},$$

 and $M_2$ is another square matrix of order $2k-2$ given by

  $$M_2=\begin{bmatrix}
x&-2s_2&-x &0&0& \ldots & 0&0 \\[1mm]
0&x &2 t_2 &-x&0&\ldots &0&0\\[1mm]
0&0&x&-2s_3 &-x& \ldots &0&0\\[1mm]
0&0&0&x&2t_3&\ldots&0&0\\[1mm]
0&0&0&0&x& \ldots &0&0 \\[1mm]
& & & & & \ddots \\[1mm]
0&0&0&0&0&\ldots&x&-2s_k\\[1mm]
-x&0 &0 &0 &0 & \ldots &0&x
\end{bmatrix}.$$
Clearly, $$\det M_1=(-1)^k\det V_x-x^{2k-1}$$
and $$\det M_2=x^{2k-2}+(-1)^{k-1}\,x\,\det U_x.$$
From equation (\ref{Chain_cp_eq3}), we obtain
\begin{equation}
\label{Chain_cp_eq4}
 \psi_{Q_{J-2A}}(x)=\frac{x}{2} \Big[2x^{2k-1}+(-1)^{k-1}x^2\det U_x+(-1)^k\det V_x\Big].
\end{equation}
Now by Corollary 2.1, any eigenvalue of $S$ which is not an eigenvalue of $Q_S$ must be equal to $-1$. Therefore by (\ref{Chain_cp_eq0}), we obtain,
 $$\psi_S(x)= \frac{(x+1)}{2}^{n-2k+1} \Big[2(x+1)^{2k-1}+(-1)^{k-1}(x+1)^2\det U_{x+1}+(-1)^k\det V_{x+1}\Big]. $$
 Which completes the proof.
\end{proof}

\begin{example}
Let $G$ be the chain graph with binary string $b=0^11^20^21^3$. Then
$$\det U_{x+1}=4,$$
and
$$\det V_{x+1}=\begin{vmatrix}
4& x+1 &0 \\[1mm]
x+1 &4 & x+1 \\[1mm]
0& x+1 &8
\end{vmatrix}=128-12(x+1)^2.$$
Therefore, Seidel characteristic polynomial of $G$ is
\begin{eqnarray*}
\begin{split}
\psi_S(x)&=\dfrac{(x+1)^5}{2}\Big[2(x+1)^3-(x+1)^2\det U_{x+1}+\det V_{x+1}\Big]\\[2mm]
&=(x+1)^8-8(x+1)^7+64(x+1)^5\\[1mm]
&=x^8-28x^6-48x^5+110x^4+416x^3+500x^2+272x+57.
\end{split}
\end{eqnarray*}
\end{example}

 In the Theorem \ref{Chain_cp_th1}, the expression of characteristic polynomial contains determinant of two symmetric matrices. The determinant of a symmetric tridiagonal matrix can be calculated iteratively. Let $B$ be a symmetric tridiagonal matrix of the form
 $$B=\begin{bmatrix}
a_1&b_1&0&0  &\ldots & 0&0 \\[1mm]
b_1 &a_2 &b_2&0& \ldots &0&0\\[1mm]
& & & &  \ddots \\[1mm]
0&0&0&0&\ldots&a_{n-1}&b_{n-1}\\[1mm]
0&0 &0  &0 & \ldots &b_{n-1}&a_n
\end{bmatrix}.$$
If $D_i$ denotes the determinant of the principal submatrix by taking first $i$ rows and columns, then $\det\ B$ can be obtained by using the following formula
\begin{equation}
\label{Chain_cp_eq5}
D_i=a_iD_{i-1}-b_{i-1}^2D_{i-2},
\end{equation}
where $D_0=1$ and $D_1=a_1$. Further let $a_1=a_2=\cdots=a_n=c$ and $b_1=b_2=\cdots=b_{n-1}=1$. If $D_i(c)$ denotes the determinant of the principal submatrix by taking first $i$ rows and columns, then the determinants $D_i(c)$, $i\geq0$ can be computed (see \cite[Theorem 3.1]{Qi}) by
\begin{equation}
\label{Chain_cp_eq6}
D_i(c)=\begin{cases}
\dfrac{\alpha^{i+1}-\beta^{i+1}}{\alpha-\beta} &\text{if }c\neq\pm2,\\[3mm]
i+1 &\text{if }c=2,\\[2mm]
(-1)^i(i+1) &\text{if }c=-2,
\end{cases}
\end{equation}
where $\alpha=\frac{1}{\beta}=\displaystyle{\frac{c+\sqrt{c^2-4}}{2}}$.

\begin{remark}[Determinant] \label{Chain_det_rm1} {\rm Putting $x=1$ in equation (\ref{Chain_cp_eq4}), we have
\begin{equation}\label{Chain_det_eq1}
\det Q_S=\frac{1}{2}\Big[2+(-1)^{k}(\det V-\det U)\Big],
\end{equation}
where $U=U_1$ and $V=V_1$.
We can find out the determinant of the tridiagonal matrices $U$ and $V$ by using the recurrence relation (\ref{Chain_cp_eq5}). Therefore the determinant of the Seidel matrix is given by,
\begin{equation}
\label{Chain_det_eq2}
\det S=(-1)^{n-2k}\det Q_S.
\end{equation}}
\end{remark}

\begin{corollary} Let $n=2k$, that is, $s_i=t_i=1$ for $i=1,\,2,\ldots,\,k$. Then
$$\det S=(-1)^{\frac{n}{2}}\,n+1.$$
\end{corollary}

\begin{proof} By (\ref{Chain_det_eq2}), $\det S=\det Q_S$. Since $s_i=t_i=1$, by (\ref{Chain_cp_eq6}), we obtain
$$\det U=D_{2k-3}(2)=2k-2,$$
and
$$\det V=4D_{2k-2}(2)-D_{2k-3}(2)=6k-2.$$
Therefore, from (\ref{Chain_det_eq1}), we get
\begin{eqnarray*}
\det S&=&\frac{1}{2}\,\Big[2+(-1)^{k-1}\det U+(-1)^k\det V\Big]\\[2mm]
&=& (-1)^{\frac{n}{2}}\,n+1.
\end{eqnarray*}
Which completes the proof.
\end{proof}


\section{Seidel energy of chain graph}
   In this section we give some energy bounds for Seidel matrix of chain graphs. We obtain a better bound of $SE(G)$ than Haemers \cite{Haemers} bound. First we recall some useful and important results.

\begin{lemma} {\rm \cite{Fan}} \label{energy_Lemma_Fan}
If $X,~Y,~Z$ be three real symmetric matrices of order $n$ such that $Z=X+Y$, then
$$E(Z)\leq E(X)+E(Y),$$
where $E(X)=\sum_{i=1}^{n}|\lambda_i(X)|$ is the energy of $X$, and $\lambda_i(X)$ $(i=1,2,\ldots,n)$ are the eigenvalues of $X$.
\end{lemma}

\begin{lemma} {\rm \cite{Biler}} \label{Lemma_Biler}
For positive real numbers $a_1,~a_2,~a_3,\cdots , a_n$,\\
$$p_1 \geq p_2^{\frac{1}{2}}\geq p_3^{\frac{1}{3}} \geq \cdots \geq p_n^{\frac{1}{n}},$$
where $p_k$ is the average of products of $k-$element subset of the set $\{a_1,a_2,a_3,\ldots, a_n \}$, that is,
\begin{align*}
p_1&=\frac{1}{n}(a_1 +a_2 +a_3 + \cdots +a_n), \\
p_2&=\frac{1}{\frac{n(n-1)}{2}}(a_1a_2  +a_1a_3 + \cdots + a_1a_n +a_2a_3 + \cdots + a_{n-1}a_n),\\
\vdots \\
p_n&=a_1a_2 \cdots a_n.
\end{align*}
Moreover, equalities hold if and only if $a_1=a_2=\cdots=a_n$.
\end{lemma}

\begin{lemma} \label{1w2} Let $\lambda_1,\,\lambda_2,\ldots,\,\lambda_{2k}$ be the eigenvalues of $Q_S$. Then $|\lambda_1|=|\lambda_2|=\cdots=|\lambda_{2k}|$ if and only if $G\cong K_2$.
\end{lemma}

\begin{proof} First we assume that $|\lambda_1|=|\lambda_2|=\cdots=|\lambda_{2k}|$. Since $\lambda_i$ $(1\leq i\leq 2k)$ are the eigenvalues of $Q_S$, then by Theorem \ref{Chain_quotient_th1}, there is an eigenvalue, say, $\lambda_j=-1$. By Theorem \ref{1w1}, we have $\lambda_1\geq\frac{n}{2}+1-\frac{2k}{n}$. If $n\geq 3$, then $|\lambda_1|\geq 1.5>1=|\lambda_j|$, a contradiction as $|\lambda_1|=|\lambda_2|=\cdots=|\lambda_{2k}|$. Otherwise, $n=2$ and hence $G\cong K_2$.

\vspace*{2mm}

Conversely, one can easily see that $|\lambda_1|=|\lambda_2|$ for $K_2$. 
\end{proof}
We are now at a position to establish the following bounds of Seidel energy of a chain graph.
\begin{theorem} \label{Chain_Energy_Th1} Let $G$ be a chain graph of order $n$ with $|\pi|=2k$. Then
$$SE(G) \leq n-2k+\sqrt{2k(n^2-2n+2k)}$$
with equality if and only if $G\cong K_2$.
\end{theorem}

\begin{proof} The multiplicity of the Seidel eigenvalue $-1$ is $n-2k+1$, in which $Q_S$ contributes one simple eigenvalue $-1$. Therefore we can write the Seidel energy in the following form:
\begin{equation}\label{Chain_Energy_eq1}
 SE(G)=n-2k+\sum\limits^{2k}_{i=1}\,|\lambda_i|,
\end{equation}
where $\lambda_1,\,\lambda_2,\ldots,\,\lambda_{2k}$ are the eigenvalues of $Q_S$. By Lemma \ref{Chain_Quotient_Lm1}, we obtain
\begin{equation}
\sum\limits^{2k}_{i=1}\,\lambda^2_i=trace\ Q_S^2=n^2-2n+2k.\label{1pm1}
\end{equation}
Setting $n=2k$ and $a_i=|\lambda_i|$ $(i=1,\,2,\ldots,\,2k)$ in Lemma \ref{Lemma_Biler}, we obtain
\begin{equation}
\left(\frac{\sum\limits^{2k}_{i=1}\,|\lambda_i|}{2k}\right)^2\geq \frac{1}{\displaystyle{\frac{2k(2k-1)}{2}}}\,\sum\limits_{i<j}\,|\lambda_i|\,|\lambda_j|,\label{1w3}
\end{equation}
that is,
\begin{align*}
(2k-1)\,\left(\sum\limits^{2k}_{i=1}\,|\lambda_i|\right)^2&\geq 4k\,\sum\limits_{i<j}\,|\lambda_i|\,|\lambda_j|,\\
&=2k\,\left[\left(\sum\limits^{2k}_{i=1}\,|\lambda_i|\right)^2-\sum\limits^{2k}_{i=1}\,\lambda^2_i\right],
\end{align*}
that is,
  $$\left(\sum\limits^{2k}_{i=1}\,|\lambda_i|\right)^2\leq 2k\,\sum\limits^{2k}_{i=1}\,\lambda^2_i=2k\,trace\ Q_S^2=2k\,(n^2-2n+2k),$$
that is,
 $$\sum\limits^{2k}_{i=1}\,|\lambda_i|\leq \sqrt{2k\,(n^2-2n+2k)}.$$
Combining the above result with equation (\ref{Chain_Energy_eq1}), we obtain the required result.

\vspace*{3mm}

Moreover, the equality holds if and only if the equality holds in (\ref{1w3}), that is, if and only if $|\lambda_1|=|\lambda_2|=\cdots=|\lambda_{2k}|$, that is, if and only if $G\cong K_2$, by Lemma \ref{1w2}.
\end{proof}

\begin{remark} {\rm Our upper bound in Theorem \ref{Chain_Energy_Th1} is better than the upper bound given in (\ref{Haemers}), that is,
$$n-2k+\sqrt{2k\,(n^2-2n+2k)}\leq n\,\sqrt{n-1},$$
that is,
 $$2kn^2\leq n^3-2\,(n-2k)\,n\,\sqrt{n-1},$$
that is,
$$2\,(n-2k)\,n\,\sqrt{n-1}\leq n^2\,(n-2k),$$
that is,
$$(n-2)^2\geq 0,$$
which is always true.}
\end{remark}

\begin{theorem}\label{Chain_Energy_Th2} Let $G$ be a chain graph of order $n$ with $|\pi|=2k$. Then
$$SE(G) \geq n-2k+\sqrt{n^2-2n+2k+2k(2k-1)|\det Q_S|^{\frac{1}{k}}}. $$
\end{theorem}

\begin{proof} Since $\lambda_1,\,\lambda_2,\ldots,\,\lambda_{2k}$ are the eigenvalues of $Q_S$, we have $\prod^{2k}_{i=1}\,|\lambda_i|=|\det Q_S|$. Setting $n=2k$ and $a_i=|\lambda_i|$ $(i=1,\,2,\ldots,\,2k)$ in Lemma \ref{Lemma_Biler}, we obtain
\begin{equation}
\frac{1}{\displaystyle{\frac{2k(2k-1)}{2}}}\,\sum\limits_{i<j}\,|\lambda_i|\,|\lambda_j|\geq \left(\prod^{2k}_{i=1}\,|\lambda_i|\right)^{\frac{1}{k}},\label{1w4}
\end{equation}
that is,
   $$2\,\sum\limits_{i<j}\,|\lambda_i|\,|\lambda_j|\geq 2k(2k-1)\,|\det Q_S|^{\frac{1}{k}},$$
that is,
  $$\left(\sum\limits^{2k}_{i=1}\,|\lambda_i|\right)^2-\sum\limits^{2k}_{i=1}\,\lambda^2_i\geq 2k(2k-1)\,|\det Q_S|^{\frac{1}{k}},$$
that is,
   $$\sum\limits^{2k}_{i=1}\,|\lambda_i|\geq \sqrt{n^2-2n+2k+2k(2k-1)\,|\det Q_S|^{\frac{1}{k}}},$$
by (\ref{1pm1}). Combining this result with equation (\ref{Chain_Energy_eq1}), we can get the required result.

\vspace*{3mm}

Moreover, the equality holds if and only if the equality holds in (\ref{1w4}), that is, if and only if $|\lambda_1|=|\lambda_2|=\cdots=|\lambda_{2k}|$, that is, if and only if $G\cong K_2$, by Lemma \ref{1w2}.
\end{proof}

\begin{remark} {\rm This lower bound is better than Haemers' lower bound given in (\ref{Haemers}). For example, consider a chain graph $0^11^20^11^1$ on $5$ vertices. Using Theorem \ref{Chain_Energy_Th2} and the determinant formula given in Remark \ref{Chain_det_rm1}, we obtain $E(S)\geq 8.78$. On the other hand, inequality (\ref{Haemers}) gives $E(S)\geq 8$.}
\end{remark}
We now establish another upper bound for the Seidel energy of  chain graphs.
\begin{theorem}\label{Chain_Energy_Th3} Let $G$ be a chain graph of order $n$ with $|\pi|=2k$. Then
$$ SE(G)\leq n(3+\sqrt{k})-2.$$
\end{theorem}

\begin{proof} Seidel matrix $S$ of $G$ is given by
$$S(G)=J-I-2A(G).$$
where $J$ is all $1$ matrix, $I$ is the identity matrix, and $A(G)$ is the adjacency matrix of $G$.
Using Lemma \ref{energy_Lemma_Fan}, we can write
\begin{equation}
\label{Chain_Energy_eq5}
SE(G)\leq E(J-I)+2E(A(G))=2(n-1)+2E(A(G)),
\end{equation}
where $E(J-I)=E(A(K_n))=\sum\limits^n_{i=1}\,|\lambda_i(A(K_n))|=2(n-1)$ and $E(A(G))=\sum\limits^n_{i=1}\,|\lambda_i(A(G))|$.
From \cite{Andelic 3}, we obtain
      $$E(A(G))\leq \frac{n}{\sqrt{8}}(\sqrt{2k}+\sqrt{2}).$$
Using this result in inequality (\ref{Chain_Energy_eq5}), we obtain
$$SE(G)\leq 2(n-1)+ \frac{2n}{\sqrt{8}}(\sqrt{2k}+\sqrt{2}),$$
which gives the required result.
\end{proof}

We now consider a special case of chain graph. We call a chain graph symmetric if $s_i=t_i=p$ for $1\leq i\leq k$ (for this case only the quotient matrix of the corresponding chain graph is a symmetric matrix). We have calculated the characteristic polynomial and energy of Seidel matrix $S$ for $k=1,\,2,\,3$ in the following example.
\begin{example}[Symmetric chain graphs] {\rm
\textbf{For k=1}, the binary string of $G$ is $b=0^p1^p$.
The Seidel eigenvalues are $-1^{2p-1},~ 2p-1$ and the characteristic polynomial is
$$\psi_S(x)=(x+1)^{2p-1}(x+1-2p),$$
and $$SE(G)=2(2p-1).$$

\noindent
\textbf{For k=2}, the binary string is $b=0^p1^p0^p1^p$. The Seidel eigenvalues are $-1^{4p-3},~ 2p-1,~ p-1\pm p \sqrt 5 $ and the characteristic polynomial is
$$\psi_S(x)=(x+1)^{4p-3}(x+1-2p)\big[(1+x)^2-2p(x+1)-4p^2\big],$$
and
$$SE(G)=2(3+\sqrt{5})p-6.$$

\noindent
\textbf{For k=3}, the binary string of $G$ is $b=0^p1^p0^p1^p0^p1^p$. The Seidel
characteristic polynomial of $G$ is
$$\psi_S(x)=(x+1)^{6p-5}\big[(1+x)^2-2p(x+1)-4p^2\big]\big[(x+1)^3-4p(x+1)^2-4p^2(x+1)+8p^3\big],$$
and the Seidel energy of $G$ is given by
$$SE(G)=(10+2\sqrt{5})p-6-2\alpha,$$
where $\alpha$ is the only negative root of the equation
$$x^3-4px^2-4p^2x+8p^3=0.$$}
\end{example}

\vspace*{3mm}

Let $\Gamma_{n,k}$ be a class of chain graphs of order $n$ with the binary string $b=0^{s_1} \underbrace{1^{1} 0^{1} 1^{1} 0^{1}\ldots 0^{1}}_{2k-2} 1^{t_k}$ such that $s_1+t_k=n-2k+2$ with $1\leq s_1\leq n-2k+1$ and $1\leq t_k\leq n-2k+1$.
\begin{lemma}\label{Chain_Energy_Lm1}
Every chain graph in $\Gamma_{n,k}$ has the same spectrum.
\end{lemma}

\begin{proof} Let $G\in\Gamma_{n,k}$. Then $b=0^{s_1} \underbrace{1^{1} 0^{1} 1^{1} 0^{1}\ldots 0^{1}}_{2k-2} 1^{t_k}$ is the binary string of chain graph $G$. Then the multiplicity of $-1$ is $n-2k+1$ and $-1$ is a simple eigenvalue of $Q_S$.
 Let $\lambda\neq -1$ be an eigenvalue of $Q_S$ and let $X=\left[\begin{array}{ccccc}
x_1&x_2&x_3&\cdots&x_{2k}
\end{array}\right]^T$ be corresponding eigenvector. Then $Q_S X=\lambda X$, that is,
\begin{equation}
\tag{B}
\begin{aligned}
(s_1 -1)x_1-x_2+x_3 -x_4+\cdots+x_{2k-1}-t_kx_{2k}&=\lambda x_1,\\
-s_1x_1+x_3+x_4+\cdots+x_{2k-1} +t_kx_{2k}&=\lambda x_2,\\
s_1x_1+x_2 -x_4+\cdots+x_{2k-1} -t_kx_{2k}&=\lambda x_3,\\
-s_1x_1+x_2-x_3+\cdots+x_{2k-1} +t_kx_{2k}&=\lambda x_4,\\
\ldots~~\ldots~~\ldots~~\ldots~~~\ldots~~\ldots~~~~~~~~~&~~\vdots\\
s_1x_1+x_2+x_3+x_4+\cdots+x_{2k-2}-t_kx_{2k}&=\lambda x_{2k-1},\\
-s_1x_1+x_2-x_3+x_4+\cdots-x_{2k-1} +(t_k-1)x_{2k}&=\lambda x_{2k}.
\end{aligned}
\end{equation}
Since $\lambda\neq -1$, from the first and last equations, we obtain
$$x_1=-x_{2k}.$$
Using $x_1=-x_{2k}$ and $s_1+t_k=n-2k+2$ in the system of equations (B), we obtain the following system of $2k-1$ equations:
\begin{equation}
\tag{C}
\begin{aligned}
(n-2k+1)x_1-x_2+x_3 -x_4+\ldots+x_{2k-1}&=\lambda x_1,\\
-(n-2k+2)x_1+x_3+x_4+\ldots +x_{2k-1}&=\lambda x_2,\\
(n-2k+2)x_1+x_2 -x_4+\ldots +x_{2k-1}&=\lambda x_3,\\
-(n-2k+2)x_1+x_2-x_3+\ldots +x_{2k-1}&=\lambda x_4,\\
\ldots~~\ldots~~\ldots~~\ldots~~~\ldots~~\ldots~~~~~~~~~~~~&~~\vdots\\
(n-2k+2)x_1+x_2+x_3+x_4+\ldots +x_{2k-2}&=\lambda x_{2k-1}.\end{aligned}
\end{equation}
The above system of equations (C) is independent of $s_1$ and $t_k$. Therefore system provides the same values of $\lambda$  for every choice of $s_1$ and $t_k$ with $s_1+t_k=n-2k+2$. Therefore every chain graph in $\Gamma_{n,k}$ has the same spectrum.
\end{proof}

\begin{corollary} \label{1w5} Every chain graph in $\Gamma_{n,k}$ has the same Seidel energy.
\end{corollary}

\begin{lemma}\label{1kank1} Let $H\in \Gamma_{n,2}$. Then
  $$Spec_S(H)=\left\{\frac{n-4\pm\sqrt{n^2+4n-12}}{2},\,1,-1^{n-3}\right\}.$$
\end{lemma}

\begin{proof} Since $H\in \Gamma_{n,2}$, the characteristic polynomial of $S(H)$ is
\begin{align*}
\psi_S(x)&=(x+1)^{n-3}\,\Big[x^3-(s_1+t_2-1)\,x^2-(2s_1+2t_2+1)\,x+3s_1+3t_2-1\Big]\\
&=(x+1)^{n-3}\,(x-1)\,\Big[x^2-(s_1+t_2-2)\,x-3s_1-3t_2+1\Big]\\
&=(x+1)^{n-3}\,(x-1)\,\Big[x^2-(n-4)\,x-3n+7\Big],
\end{align*}
which gives the required result.
\end{proof}

\begin{theorem} \label{1kank2} Let $G$ be a chain graph of order $n$ with the binary string $b=0^{s_1} 1^{t_1} 0^{s_2} 1^{t_2}$ $(s_1+s_2+t_1+t_2=n)$. Then $SE(G)\geq n-2+\sqrt{n^2+4n-12}$ with equality if and only if $G\in \Gamma_{n,2}$.
\end{theorem}

\begin{proof} First we assume that $G\in \Gamma_{n,2}$. Then by Lemma \ref{1kank1}, we obtain
   $$SE(G)=n-2+\sqrt{n^2+4n-12}$$
and hence the equality holds.

\vspace*{3mm}

\begin{figure}[ht!]
\begin{center}
\begin{tikzpicture}[scale=0.8,style=thick]
\def\vr{8pt} 
\draw (0,0)  [fill=black] circle (\vr);
\draw (3,0)  [fill=black] circle (\vr);
\draw (6,0)  [fill=black] circle (\vr);
\draw (8.5,1)  [fill=black] circle (\vr);
\draw (8.5,-1)  [fill=black] circle (\vr);
\draw (4,-1.5) node {$G_1$};
\draw (0,0)-- (3,0);
\draw (3,0)-- (6,0);
\draw (6,0)-- (8.5,1);
\draw (6,0)-- (8.5,-1);
\end{tikzpicture}
\end{center}
\vspace{-0.3cm}
\caption{Graph $G_1$.}\label{fig: T-k}\vspace{-0.7cm}
 \end{figure}
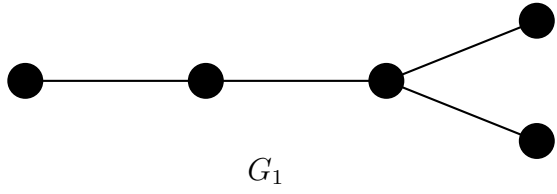

 \vspace*{8mm}

Next we assume that $G\notin \Gamma_{n,2}$. Then $2\leq s_1+t_2\leq n-3$ and $n\geq 5$. For $n=5$, $G\cong G_1$ (see, Fig. 1). One can easily obtain that
  $$Spec_S(G_1)=\left\{3,\,1.56,\,-1^2,\,-2.56\right\}.$$
Hence
$$SE(G_1)>9.12>8.75>n-2+\sqrt{n^2+4n-12}.$$

\vspace*{3mm}

\noindent
Otherwise, $n\geq 6$. Now, the characteristic polynomial of $G$ is given by
\begin{equation}
    \psi_S(x)=(x+1)^{n-3}\,\Big[(x+1)^3-n(x+1)^2+4s_2t_1(s_1+t_2)\Big].
\end{equation}
Let $\lambda_1\geq \lambda_2\geq \lambda_3$ be the roots of the following equation:
   $$f(x)=0,$$
where
\begin{equation}
f(x)=x^3-(n-3)\,x^2-(2n-3)\,x-n+1+4s_2t_1(s_1+t_2).\label{bm1}
\end{equation}
Then $\psi_S(x)=(x+1)^{n-3}\,f(x)$. Moreover, we obtain
\begin{align}
\lambda_1+\lambda_2+\lambda_3&=n-3,\label{1kcd1}\\
\lambda_1\,\lambda_2+\lambda_2\,\lambda_3+\lambda_3\,\lambda_1&=-(2n-3),\label{1kcd2}\\
\lambda_1\,\lambda_2\,\lambda_3&=n-1-4s_2t_1(s_1+t_2).\label{1kcd3}
\end{align}

\noindent
{\bf Claim 1.} $\lambda_1\geq \lambda_2>1$ and $\lambda_3<0$.

\vspace*{3mm}

\noindent
{\bf Proof of Claim 1.} Let us consider a function
     $$h(x,y)=xy\,(n-x-y),$$
where $x$ are $y$ are positive integers such that $3\leq x+y\leq n-2$ and $x\geq y\geq 1$. Since $x+y\geq 3$, we have $x\geq 2$. We have to prove that \begin{equation}
xy\,(n-x-y)\geq 2\,(n-3).\label{1kan2}
\end{equation}
First we assume that $y=1$. Then $2\leq x\leq n-3$. One can easily check that $g(x)=x\,(n-x-1)$ is an increasing function on $x\leq \frac{n-1}{2}$ and a decreasing function on $x\geq \frac{n-1}{2}$. Then $xy\,(n-x-y)=x\,(n-x-1)\geq \min\{g(2),\,g(n-3)\}=2\,(n-3)$. Hence (\ref{1kan2}) holds.

\vspace*{2mm}

Next we assume that $y\geq 2$. If $x+y<\frac{n+1}{2}$, then $xy\,(n-x-y)>2\,(n-3)$ as $x\geq 2$ and $y\geq 2$. Again (\ref{1kan2}) holds. Otherwise, $x+y\geq \frac{n+1}{2}$. Then $x\geq \frac{n+1}{4}$ as $x\geq y$. For $x+y\leq n-4$, we have
    $$xy\,(n-x-y)\geq 2\,(n+1)>2\,(n-3)~~\mbox{ as }y\geq 2~\mbox{ and }n-x-y\geq 4,$$
(\ref{1kan2}) holds. For $x+y=n-3$, we have
$$xy\,(n-x-y)=3x(n-x-3)\geq 6(n-5)\geq 2\,(n-3)$$
as $n\geq 6$,$\frac{n-3}{2}\leq x\leq n-5$. Again (\ref{1kan2}) holds. For $x+y=n-2$,
 $$xy\,(n-x-y)=2x\,(n-x-2)\geq 4(n-4)>2\,(n-3)$$
as $\frac{n-2}{2}\leq x\leq n-4$ and $n\geq 6$. Hence (\ref{1kan2}) holds.

\vspace*{3mm}

Since $s_1+s_2+t_1+t_2=n$, we obtain $s_2t_1(s_1+t_2)=s_2t_1\,(n-s_2-t_1)\geq 2\,(n-3)$, by (\ref{1kan2}). Using this result in (\ref{1kcd3}), we obtain $\lambda_1\,\lambda_2\,\lambda_3<0$ as $n\geq 6$. Therefore $\lambda_i<0$ for $1\leq i\leq 3$, or $\lambda_1\geq \lambda_2>0$ and $\lambda_3<0$. By (\ref{1kcd1}), $\lambda_i<0$ $(1\leq i\leq 3)$ is not possible. Thus we have $\lambda_1\geq \lambda_2>0$ and $\lambda_3<0$. From (\ref{bm1}), we obtain $f(x)\rightarrow +\infty$ as $x\rightarrow \infty$ and
  $$f(1)=-4n+8+4s_2t_1(s_1+t_2)=4\,\Big[-(n-2)+s_2t_1\,(n-s_2-t_1)\Big]\geq 4\,(n-4)>0$$
as $s_1+s_2+t_1+t_2=n$ and by (\ref{1kan2}). These results with (\ref{1kcd1}), we obtain $\lambda_1\geq \lambda_2>1$. This proves {\bf Claim 1}.

\vspace*{3mm}

\noindent
{\bf Claim 2.} $\lambda_1\,\lambda_2>\displaystyle{\frac{n-4+\sqrt{n^2+4n-12}}{2}}$.

\vspace*{3mm}

\noindent
{\bf Proof of Claim 2.} We consider the following two cases:

\vspace*{2mm}

\noindent
${\bf Case\,1.}$ $\lambda_3\leq -3$. We have $\lambda_1+\lambda_2=n-3-\lambda_3\geq n$. Since $f(x)=x\,(n-x)$ is an increasing function on $x\leq n/2$ and a decreasing function on $x\geq n/2$, we obtain
                    $$\lambda_1\,\lambda_2\geq (n-\lambda_2)\,\lambda_2\geq n-1>\frac{n-4+\sqrt{n^2+4n-12}}{2}$$
as $1<\lambda_2<\frac{n-3}{2}$.

\vspace*{3mm}

\noindent
${\bf Case\,2.}$ $-3<\lambda_3<0$. Since $s_1+t_2\leq n-3$, we have $3\leq s_2+t_1\leq n-2$. Since $s_2t_1(s_1+t_2)=s_2t_1\,(n-s_2-t_1)\geq 2(n-3)$, from (\ref{1kcd3}), we obtain
   $$\lambda_1\,\lambda_2\,\lambda_3=n-1-4s_2t_1(s_1+t_2)\leq -7n+23,$$
that is,
   $$\lambda_1\,\lambda_2\geq \frac{-7n+23}{\lambda_3}\geq n-1>\frac{n-4+\sqrt{n^2+4n-12}}{2}$$
as $n\geq 6$. Which proves {\bf Claim 2}.

\vspace*{3mm}

\noindent
By {\bf Claims 1} and {\bf 2} with (\ref{1kcd2}), we obtain
\begin{align*}
(|\lambda_1|+|\lambda_2|+|\lambda_3|)^2&=(\lambda_1+\lambda_2-\lambda_3)^2\\[1mm]
&=(\lambda_1+\lambda_2+\lambda_3)^2-4\,(\lambda_1\lambda_2+\lambda_2\lambda_3+\lambda_3\lambda_1)+4\,\lambda_1\lambda_2\\[1mm]
&=(n-3)^2+4\,(2n-3)+2\,[n-4+\sqrt{n^2+4n-12}]\\[1mm]
&>n^2+4n-11+2\,\sqrt{n^2+4n-12}=(\sqrt{n^2+4n-12}+1)^2,
\end{align*}
that is,
   $$|\lambda_1|+|\lambda_2|+|\lambda_3|>\sqrt{n^2+4n-12}+1.$$
Since
    $$Spec_S(G)=\Big\{\lambda_1,\,\lambda_2,\,\lambda_3,\,-1^{n-3}\Big\},$$
we obtain
$$SE(G)=n-3+|\lambda_1|+|\lambda_2|+|\lambda_3|>n-2+\sqrt{n^2+4n-12}.$$
This completes the proof of the theorem.
\end{proof}

From the motivation of the above theorem with Corollary \ref{1w5}, we state the following conjecture:
\begin{conjecture} Let $G$ be a chain graph of order $n$ with the binary string $b=0^{s_1} 1^{t_1} 0^{s_2} \ldots 0^{s_k} 1^{t_k}$ such that $\sum\limits^k_{i=1}\,(s_i+t_i)=n,\,s_i\geq 1\,(1\leq i\leq k)$, and $t_i\geq 1\,(1\leq i\leq k)$. Then $SE(G)\geq SE(H)$, where $H\in \Gamma_{n,k}$. Moreover, the equality holds if and only if $G\in \Gamma_{n,k}$.
\end{conjecture}

\section*{Acknowledgements}
S. Mandal thanks to University Grants Commission, India for financial support under the beneficiary code BININ01569755. K. C. Das is supported by National Research Foundation funded by the Korean government (Grant No. 2021R1F1A1050646).



\begin{thebibliography}{99}
\bibitem{Akbari 1}
S. Akbari, M. Einollahzadeh, M. M. Karkhaneei, M. A. Nematollahi, \textit{Proof of a conjecture on the Seidel energy of Graphs}, European J. Combin. 86 (2020) 103078, 8 pp.


\bibitem{Akbari 2}
S. Akbari, J. Aksari, K. C. Das, \textit{Some properties of eigenvalues of the Seidel matrix}, Linear  Multilinear Algebra (2020), in press. DOI: doi.org/10.1080/03081087.2020.1790481

\bibitem{AAD} A. Alazemi, M. An\dj eli\'c, K. C. Das, C. M. D. Fonseca, Chain graph sequences and Laplacian spectra of chain graphs, Linear  Multilinear Algebra (2022), in press.

\bibitem{Andelic 1}
A. Alazemi, M. An\dj eli\'c, S. K. Simi\'c, \textit{Eigenvalue location for chain graphs}, Linear Algebra Appl. 505 (2016) 194--210.

\bibitem{Andelic 2}
M. An\dj eli\'c, S. K. Simi\'c, D. Zivkovi\'c, E. Dolicanin, \textit{Fast algorithm for computing the characteristic polynomial of threshold and  chain graphs}, Appl. Math. Comput. 332 (2018) 329--337.

\bibitem{AID} J. Askari, A. Iranmanesh, K. C. Das, \textit{Seidel-Estrada Index}, Jour. Ineq. Appl. (2016) 2016: 120.

\bibitem{Bell 1}
F. K. Bell, D. Cvetkovi\'c, P. Rowlinson, S. K. Simi\'c, \textit{Graphs for which least eigenvalue is minimal, I}, Linear Algebra Appl. 429 (2008) 234--241.

\bibitem{Bell 2}
F. K. Bell, D. Cvetkovi\'c, P. Rowlinson, S. K. Simi\'c, \textit{Graphs for which least eigenvalue is minimal, II}, Linear Algebra Appl. 429 (2008) 2168--2179.

\bibitem{Berman}
A. Berman, N. S. Monderer, R. Singh, X. D. Zhang, \textit{Complete multipartite graphs that are determined, up to switching, by their Seidel spectrum}, Linear Algebra Appl. 564 (2019) 58--71.

\bibitem{Bhattacharya}
A. Bhattacharya, S. Friedland, U. N. Peled, \textit{On the first eigenvalue of bipartite graphs}, Electron. J. Combin. 15 (1) (2008) \#144, 23pp.

\bibitem{Biler}
P. Biler, A. Witkowski, \textit{Problems in Mathematical Analysis}, Chapman and Hall, New York, 1990.

\bibitem{Brouwer}
A. E. Brouwer, W. H. Haemers, \textit{Spectra of Graphs}, Springer, New York, 2012.


\bibitem{Andelic 3}
K. C. Das, A. Alazemi, M. An\dj eli\'c, \textit{On energy and Laplacian energy of  chain graphs}, Discrete Appl. Math. 284 (2020) 391--400.

\bibitem{Das 1}
K. C. Das, I. Gutman, \textit{Bounds for the energy of graphs}, Hacet. J. Math. Stat. 45 (3) (2016) 695--703.

\bibitem{Fan}
K. Fan, \textit{Maximum properties and inequalities for the eigenvalues of completely continuous operators}, Proc. Nat. Acad. Sci. USA. 37 (1951) 760--766.

\bibitem{Ghorbani 1}
E. Ghorbani, \textit{On eigenvalues of Seidel matrices and Haemers' conjecture}, Des. Codes Cryptogr. 84 (1-2) (2017) 189--195.

\bibitem{Godsil}
C. Godsil, G. Royle, \textit{Algebraic Graph Theory}, Springer-Verlag, New York, 2001.

\bibitem{Haemers}
W. H. Haemers, \textit{Seidel Switching and Graph Energy}, MATCH Commun. Math. Comput. Chem. 68 (2012) 653--659.

\bibitem{Aksari}
A. Iranmanesh, J. A. Farsangi, \textit{Upper and lower bounds for the power of eigenvalues in Seidel matrix}, J. Appl. Math. Inform. 33 (5-6) (2015) 627--633. DOI: dx.doi.org/10.14317/jami.2015.627

\bibitem{Mandal}
S. Mandal, R. Mehatari, \textit{On the Seidel spectrum of threshold graphs}, (2021). Preprint: arXiv:2101.03364.

\bibitem{Mei}
Y. Mei, C. Guo, M. Liu, \textit{The bounds of the energy and Laplacian energy of chain graphs}, AIMS Math. 6 (5) (2021) 4847--4859. DOI: 10.3934/math.2021284







\bibitem{Oboudi 1}
M. R. Oboudi, \textit{Energy and Seidel energy of graphs}, MATCH Commun. Math. Comput. Chem. 75 (2016) 291--303.

\bibitem{Oboudi 2}
M. R. Oboudi, \textit{Seidel energy of complete multipartite graphs}, Spec. Matrices 9 (2021) 212--216. DOI: doi.org/10.1515/spma-2020-0131.

\bibitem{Qi}
F. Qi, V.  \v{C}er\v{n}anov\'a, Y. S. Semenov, S\textit{ome tridiagonal determinants related to central Delannoy numbers, the Chebyshev polynomials, and the Fibonacci polynomials}, Politehn. Univ. Bucharest Sci. Bull. Ser. A Appl. Math. Phys. 81 (2019) 123--136.

\bibitem{Ramane}
H. S. Ramane, K. Ashoka, D. Patil, \textit{On the Seidel Laplacian and Seidel signless Laplacian polynomials of graphs}, Kyungpook Math. J. 61 (1) (2021) 155--168. DOI: doi.org/10.5666/KMJ.2021.61.1.155


\bibitem{Styan}
H. Wolkowicz, G. P. H. Styan, \textit{Bounds for the eigenvalues using traces}, Linear Algebra Appl. 29 (1980) 471--506.





\end{thebibliography}
\end{document}